\documentclass[12pt]{amsart}
\usepackage{xcolor}
\usepackage{amssymb,amsmath,amsthm,amsfonts}
\usepackage{enumerate}
\input{xy}
\xyoption{all}
\usepackage{graphicx}
\usepackage{tikz-cd}
\usepackage{verbatim}
\usepackage[normalem]{ulem}
\usepackage[bookmarks,colorlinks,breaklinks]{hyperref}
\hypersetup{linkcolor=blue,citecolor=blue,filecolor=blue,urlcolor=blue} 
\usepackage{todonotes}

\newtheorem{theorem}{Theorem} [section]
\newtheorem{prop}[theorem]{Proposition}
\newtheorem{lemma}[theorem]{Lemma}

\newtheorem{conjecture}[theorem]{Conjecture}

\theoremstyle{definition}

\newtheorem{example}[theorem]{Example}
\newtheorem{remark}[theorem]{Remark}

\numberwithin{equation}{section}
\numberwithin{figure}{section}

\definecolor{myblue}{rgb}{0.6, 0.9, 1}

\makeatletter

\newcommand{\Rmnum}[1]{\expandafter\@slowromancap\romannumeral #1@}
\makeatother

\definecolor{myblue}{rgb}{0.6, 0.9, 1}
\definecolor{mygreen}{rgb}{0,0,1}
\definecolor{purple}{rgb}{0.6,0.2,1}
\definecolor{orange}{rgb}{0.8,0,0.2}

\newcommand{\bC}{\mathbb{C}}
\newcommand{\bP}{\mathbb{P}}
\newcommand{\C}{\mathbb{C}}

\newcommand{\bQ}{\mathbb{Q}}
\newcommand{\bR}{\mathbb{R}}
\newcommand{\R}{\mathbb{R}}
\newcommand{\bN}{\mathbb{N}}

\newcommand{\bA}{\mathbb{A}}

\newcommand{\Gal}{\operatorname{Gal}}
\newcommand{\Kbar}{\overline{K}}
\newcommand{\Qbar}{\overline{\bQ}}

\newcommand{\cL}{\mathcal{L}}
\newcommand{\cS}{\mathcal{S}}

\newcommand{\rank}{\operatorname{rank}}

\newcommand\iso{\simeq}

\newcommand{\Z}{\mathbb{Z}}

\newcommand{\cX}{\mathcal{X}}
\newcommand{\cY}{\mathcal{Y}}
\newcommand{\cA}{\mathcal{A}}
\newcommand{\Crit}{\mathrm{Crit}}
\newcommand\codim{\operatorname{codim}}
\newcommand{\M}{\mathrm{M}}
\newcommand{\Prep}{\mathrm{Prep}}
\newcommand{\Rat}{\mathrm{Rat}}
\newcommand{\Tbif}{T_{\mathrm{bif}}}


\begin{document}

\title[The geometry of preperiodic points, in families]{The geometry of preperiodic points \\ in families of maps on $\bP^N$ 
}

\author{Laura DeMarco and Niki Myrto Mavraki}
	\address{Department of Mathematical and Computational science \\ University of Toronto Mississauga\\ 3359 Mississauga Road \\ 
Mississauga \\ ON L5L 1C6}
	\email{myrto.mavraki@utoronto.ca} 
	\address{Department of Mathematics \\ Harvard University
\\ 1 Oxford Street \\
Cambridge \\ MA 02138}
	\email{demarco@math.harvard.edu}

\begin{abstract}
We study the dynamics of complex algebraic families of maps on $\bP^N$, and the geometry of their preperiodic points.  The goal of this article is to formulate a conjectural characterization of the subvarieties of $S \times\bP^N$ containing a Zariski-dense set of preperiodic points, where the parameter space $S$ is a quasiprojective complex algebraic variety; the characterization is given in terms of the non-vanishing of a power of the invariant Green current associated to the family of maps.  This conjectural characterization is inspired by and generalizes the Relative Manin-Mumford Conjecture for families of abelian varieties, recently proved by Gao and Habegger, and it includes as special cases the Manin-Mumford Conjecture (theorem of Raynaud) and the Dynamical Manin-Mumford Conjecture (posed by Ghioca, Tucker, and Zhang).  We provide examples where the equivalence is known to hold, and we show that many recent results can be viewed as special cases.  Finally, we give the proof of one implication in the conjectural characterization.
\end{abstract}


\maketitle

\thispagestyle{empty}

\bigskip
\section{Introduction}

Let $S$ be a smooth and irreducible quasiprojective variety defined over the field $\C$ of complex numbers.  Fix an integer $d \geq 2$.  An {\bf algebraic family of endomorphisms of $\bP^N$ of degree $d$} is a morphism
	$$\Phi: S \times \bP^N \to S \times \bP^N$$
given by $\Phi(s, z) = (s, f_s(z))$, where each $f_s$ is an endomorphism of the complex projective space $\bP^N$ of degree $d$.  Throughout this article, we let $\mathcal{X}\subset S \times\bP^N$ denote a closed irreducible subvariety which is flat over $S$.  We will use boldface $\mathbf{X}$ to denote the generic fiber of $\mathcal{X}$ and let $\mathbf{\Phi}:\mathbf{P^N}\to \mathbf{P^N}$ be the map induced by $\Phi$, viewed as an endomorphism over the function field $\bC(S)$. 

Inspired by Pink's conjectures \cite{Pink:conjecture} (especially \cite[Conjecture 6.2]{Pink:conjecture:preprint}), the recent theorems of Gao and Habegger on families of abelian varieties \cite[Theorems 1.1 and 1.3]{Gao:Habegger:RMM}, and by Zhang's proposed extensions to more general dynamical systems on projective varieties \cite{Zhang:distributions}, we discuss the following aim:  {\em to characterize the subvarieties $\mathcal{X}\subset S \times\bP^N$ which contain a Zariski-dense set of preperiodic points of $\Phi$.}

This is an ambitious goal; even the case where $S$ is a point remains vastly open, where a conjectural characterization goes by the name of the Dynamical Manin-Mumford Conjecture, proposed by Shouwu Zhang \cite{Zhang:distributions} and reformulated in \cite{Ghioca:Tucker:Zhang} and \cite{Ghioca:Tucker:DMM}.  Roughly speaking, if $\dim S = 0$, a subvariety $X$ of $\bP^N$ should contain a Zariski-dense set of preperiodic points if and only if it is itself preperiodic (or is preperiodic for an endomorphism that commutes with $\Phi$).  For $\dim S>0$, we will see that the variety $\cX \subset S \times \bP^N$ needs only be ``big enough" in a preperiodic subvariety.

To make a precise conjecture, we employ the notion of $\Phi$-special subvarieties introduced by Ghioca and Tucker in \cite{Ghioca:Tucker:DMM}.  We say that an irreducible subvariety $\mathcal{Y}\subset S \times\bP^N$, which is flat over a Zariski-open subset of $S$ and Zariski-closed in $S\times \bP^N$, is {\bf $\Phi$-special} if there exist a subvariety $\mathbf{Z}\subset \mathbf{P^N}$ over the algebraic closure $\overline{\C(S)}$ containing the generic fiber ${\bf Y}$ of $\cY$, a polarizable endomorphism $\mathbf{\Psi}:\mathbf{Z}\to \mathbf{Z}$, and an integer $n\in\bN$ so that the following hold: 
\begin{itemize}
\item $\mathbf{\Phi}^n(\mathbf{Z}) = \mathbf{Z}$;
\item $\mathbf{\Phi}^n\circ\mathbf{\Psi}=\mathbf{\Psi}\circ\mathbf{\Phi}^n$ on $\mathbf{Z}$; and
\item $\mathbf{Y}$ is preperiodic for $\mathbf{\Psi}$. 
\end{itemize} 
An endomorphism $\Psi$ of a projective variety $Z$ (over any field of characteristic 0) is {\bf polarizable} if there exist an ample line bundle $L$ on $Z$ and integer $q>1$ so that $\Psi^* L \iso L^q$.  Polarizable is equivalent to being the restriction of an endomorphism $\tilde{\Psi}: \bP^N\to \bP^N$ to an invariant subvariety, where the embedding $Z \hookrightarrow \bP^N$ may be defined by a power of $L$ and $\Psi = \tilde{\Psi}|_Z$; see, for example,  \cite{Fakhruddin:selfmaps} \cite{Meng:Zhang:polarized} \cite{Meng:Zhang:normal} for more on polarized endomorphisms.

Let $r_{\Phi, \mathcal{X}}$ denote the {\bf relative special dimension} of $\cX$ over $S$; that is, 
$$r_{\Phi,\mathcal{X}}:=\min\{\dim_S \mathcal{Y} \; : \; \mathcal{Y}~\text{ is }\Phi\text{-special and }\mathcal{X}\subset \mathcal{Y}\}$$
is the minimal relative dimension over $S$ of a $\Phi$-special subvariety containing $\mathcal{X}$, where  
	$$\dim_S \cY = \dim \cY - \dim S$$
is the dimension of a general fiber of the projection to $S$.  Note that $0 \leq \dim_S \cX \leq r_{\Phi, \cX} \leq N$. We remark that there is not necessarily a ``smallest $\Phi$-special subvariety containing $\cX$", in contrast with the setting of abelian varieties, as the intersection of $\Phi$-special subvarieties is not necessarily $\Phi$-special; see \cite[Example 3.1]{Ghioca:Tucker:DMM} for an example when $\dim S = 0$.

We need one more definition to formulate the conjecture.  The {\bf Green current} $\hat{T}_\Phi$ on $S\times \bP^N$ is defined as follows; see, for example, \cite[\S2.3]{Gauthier:Vigny}.  Let $\omega$ be the Fubini-Study form on $\bP^N$, and let $\hat{\omega}$ be the smooth $(1,1)$-form on $S \times\bP^N$ defined by pulling back $\omega$ via the projection $S \times \bP^N \to \bP^N$.  Then the sequence $d^{-n}(\Phi^n)^{*}(\hat{\omega})$ converges weakly to the closed positive $(1,1)$-current $\hat{T}_{\Phi}$ on $S\times\bP^N$; the potentials converge locally uniformly.  

\begin{conjecture}\label{characterization}
Let $\Phi: S\times\bP^N\to S \times\bP^N$ be an algebraic family of morphisms of degree $> 1$, and let $\mathcal{X}\subset S\times\bP^N$ be a complex, irreducible subvariety which is flat over $S$. The following are equivalent. 
\begin{enumerate}
\item $\mathcal{X}$ contains a Zariski-dense set of $\Phi$-preperiodic points. 
\item  
$\hat{T}_{\Phi}^{\wedge r_{\Phi,\mathcal{X}}}\wedge [\mathcal{X}]\not= 0$ for the relative special dimension $r_{\Phi, \cX}$.
\end{enumerate}
\end{conjecture}

Here, by convention, $\hat{T}^{\wedge 0} = 1$.  Note that if $\cY \subset S\times \bP^N$ is a $\Phi$-special subvariety which is invariant for an endomorphism $\Psi$ that commutes with an iterate of $\Phi$, then $\hat{T}_\Phi = \hat{T}_\Psi$ on $\cY$ because of the commuting relation.  Moreover, the slices of $\hat{T}_{\Phi}^{\wedge \dim_S\cY}$ in fibers of $\cY\to S$ are the measures of maximal entropy for the restriction of $\Psi$; see, for example, \cite{Dinh:Sibony:dynamics}.  So condition (2) of Conjecture \ref{characterization} means that (an iterate of) $\cX$ intersects these families of measures nontrivially.  We point out here that the non-vanishing of $\hat{T}_{\Phi}^{\wedge r_{\Phi,\mathcal{X}}}\wedge [\mathcal{X}]$ can also be seen from an arithmetic viewpoint in the theory of adelic metrized line bundles on quasi-projective varieties developed by Yuan--Zhang \cite{Yuan:Zhang:quasiprojective}. In the notation of \cite{Guo} and \cite{Yuan:Zhang:quasiprojective}, it is equivalent to the non-vanishing of an intersection number $\widetilde{L_\Phi}^{r_{\Phi,\mathcal{X}}}|_{X}\cdot \overline{H}^{\dim \mathcal{X}-r_{\Phi,\mathcal{X}}}$ for some $\overline{H}\in\widetilde{\mathrm{Pic}}(X/\mathbb{C})$ \cite[Theorem 1.2]{Guo}. 

Conjecture \ref{characterization} is a generalization of the recent theorem of Gao and Habegger on families of abelian varieties \cite[Theorem 1.3]{Gao:Habegger:RMM}, when $\Phi$ preserves an abelian scheme $\mathcal{A}$ over $S$, inducing a homomorphism, and $\mathcal{X}$ is contained in $\mathcal{A}$.  The preperiodic points of $\Phi$ in $\mathcal{A}$ coincide with the torsion points for the group structure.  The current $\hat{T}_{\Phi}$ restricts to a Betti form on $\mathcal{A}$, and condition (2) of Conjecture \ref{characterization} is equivalent to saying that $\mathcal{X}$ has maximal Betti rank; see Section \ref{Betti} for details.  

\begin{remark}
We note that Conjecture \ref{characterization} could be formulated verbatim for general families of polarized endomorphisms; see \cite[\S2.3]{Gauthier:Vigny} for a construction of the fibered Green current in this setting.  We chose to keep the presentation concrete, since the seemingly more general statement reduces to the one given here by Fakhruddin's observation in \cite{Fakhruddin:selfmaps}, which shows that any polarized endomorphism extends to an endomorphism of a projective space.
\end{remark}

Conjecture \ref{characterization} is known to hold when $N=1$, because it follows from \cite[Theorem 1.1]{D:stableheight}, as we explain in Section \ref{known}.  In this article, we show that many existing works and conjectures can be viewed as special cases of Conjecture \ref{characterization}, sometimes in surprising ways.  In fact, we will see that many powerful statements follow from a weaker form of Conjecture \ref{characterization}: 

\begin{conjecture}\label{dimension}
Let $\Phi: S\times\bP^N\to S \times\bP^N$ be an algebraic family of morphisms of degree $> 1$, and let $\mathcal{X}\subset S\times\bP^N$ be an irreducible, flat family of subvarieties over $S$ containing a Zariski-dense set of preperiodic points of $\Phi$.  Then $\mathcal{X}$ has codimension $\le\dim S$ in a $\Phi$-special subvariety. 
\end{conjecture}

\noindent 
Conjecture \ref{dimension} is easily obtained from Conjecture \ref{characterization}, because $\hat{T}_{\Phi}^{\wedge r_{\Phi,\mathcal{X}}}\wedge [\mathcal{X}]\not= 0$ implies that $\dim \mathcal{X} \geq r_{{\Phi,\mathcal{X}}} = \dim \mathcal{Y} - \dim S$ for a $\Phi$-special subvariety $\mathcal{Y}$ containing $\mathcal{X}$. A special case of Conjecture \ref{dimension} in a dynamical setting was proposed as \cite[Conjecture 1.9]{Mavraki:Schmidt}; see also \cite[Theorem 1.8]{Mavraki:Schmidt}.

Gao and Habegger pointed out in \cite{Gao:Habegger:RMM} that the converse assertion to Conjecture \ref{dimension} is not true, with explicit examples and constructions in the setting where $\Phi$ fixes an abelian subscheme of $S \times\bP^N$; see also \cite{Gao:generic}.  This led them to formulate their \cite[Theorem 1.3]{Gao:Habegger:RMM}, which is what our Conjecture \ref{characterization} aims to extend.   Some of those counterexample constructions arose already in the (conjectural) characterization of subvarieties with geometric canonical height 0, in the Geometric Dynamical Bogomolov Conjecture formulated by Gauthier and Vigny \cite{Gauthier:Vigny}.  The conjecture in \cite{Gauthier:Vigny} aims to extend theorems for abelian varieties over function fields of characteristic 0, proved in \cite{Gao:Habegger:Bog, CGHX}.  We discuss the relation between \cite[Conjecture 1.9]{Gauthier:Vigny} and Conjecture \ref{characterization} in Section \ref{stability}.

\begin{remark}\label{small}  
In the case that both $\Phi$ and $\mathcal{X}$ are defined over $\Qbar$, we expect Conjectures \ref{characterization} and \ref{dimension} to remain true upon replacing $\Phi$-preperiodic points by $\Phi$-small points, in the spirit of the Bogomolov Conjecture over number fields. Here a sequence of points $\{x_n\}_{n}\subset \mathcal{X}(\Qbar)$ is called $\Phi$-small if $\hat{h}_{\Phi}(x_n)\to 0$, where $\hat{h}_{\Phi}(x):=\hat{h}_{\Phi_{\pi(x)}}(x)$ is the fiber-wise Call--Silverman canonical height over $\Qbar$ introduced in \cite{Call:Silverman}, and $\pi: S\times \bP^N \to S$ is the projection. This would generalize the Relative Bogomolov Conjecture in \cite{DGH:rbc}.
\end{remark}

We conclude this article with a proof of the following implication.

\begin{theorem} \label{easy}
Condition (2) implies condition (1) of Conjecture \ref{characterization}.
\end{theorem}

\noindent 
The proof of Theorem \ref{easy} is based on the methods of Dujardin \cite{Dujardin:higherbif}, Berteloot-Bianchi-Dupont \cite{Berteloot:Bianchi:Dupont}, and Gauthier \cite{Gauthier:abscont} to study supports of bifurcation currents and measures.  For endomorphisms of abelian varieties, the implication (2) $\implies$ (1) was straightforward and observed in \cite{ACZ:Betti}, as we explain in Section \ref{Betti}.

\bigskip\noindent
{\bf Contents of this article.}   We begin Section \ref{Betti} by introducing a dynamical notion of $\Phi$-rank of $\cX$, and we compare Conjecture \ref{characterization} to the theorems of Gao and Habbeger from \cite{Gao:Habegger:RMM}.  In particular, Conjecture \ref{characterization} characterizes subvarieties $\cX \subset S \times \bP^N$ of maximal $\Phi$-rank; see Conjecture \ref{max rank}.  We also compare the conditions of Conjecture \ref{characterization} to the concept of non-degeneracy of subvarieties $\cX$, which was originally introduced for families of abelian varieties.  We then provide a brief survey in Section \ref{known} of familiar cases of Conjecture \ref{characterization}, outside of the setting of abelian schemes.  In particular, we remind the reader of the Dynamical Manin-Mumford Conjecture, first posed in \cite{Zhang:distributions}, which corresponds to Conjecture \ref{characterization} in the case where $\dim S = 0$.  We show that Conjecture \ref{characterization} is known to hold in dimension $N=1$ over any base $S$ (and is in fact equivalent to \cite[Theorem 1.1]{D:stableheight}), and we relate it to the concept of $J$-stability for maps on $\bP^1$.  We then show that a conjecture of \cite{BD:polyPCF}, and the recent classification of ``special curves" in the moduli space $\M_d^1$ of maps on $\bP^1$ in \cite{Ji:Xie:DAO}, is a special case of Conjecture \ref{characterization}; see Theorem \ref{DAO implication}.  Moreover, we illustrate the strength of Conjecture \ref{characterization} with an example showing that uniform versions of the conjecture are consequences of the conjecture itself; the example we provide in Proposition \ref{uniform shared} comes from the study of shared preperiodic points for distinct maps on $\bP^1$.  In Section \ref{further}, we show that Conjecture \ref{characterization} implies the recent sparsity theorem of Gauthier, Taflin, and Vigny in \cite{Gauthier:Taflin:Vigny} about PCF maps in the moduli spaces $\M_d^N$ of maps on $\bP^N$, for $N>1$.  We present Conjecture \ref{sparsity} as a special case of Conjecture \ref{characterization} that extends the sparsity result of \cite{Gauthier:Taflin:Vigny} to more general families of subvarieties in $\bP^N$.  We then explore the concept of minimal $\Phi$-rank and compare Conjecture \ref{characterization} to the Geometric Bogomolov Conjecture posed in \cite{Gauthier:Vigny} in Section \ref{stability}.  Finally, in Section \ref{easy section}, we provide the proof of Theorem \ref{easy}.


\bigskip\noindent{\bf Acknowledgements.}  We would like to thank Ziyang Gao, Thomas Gauthier, Lars K\"uhne, Harry Schmidt, and Gabriel Vigny for many interesting discussions and their help during the preparation of this article.  We thank the anonymous referees for helpful suggestions.  We also thank the Simons Foundation for their support during a Symposium in August 2022 where this work was initiated. This project was supported in part with funding from the National Science Foundation, the Radcliffe Institute for Advanced Study, and the Natural Sciences and Engineering Research Council of Canada.

\bigskip
\section{Dynamical rank and Betti rank}
\label{Betti}

In this section, we compare Conjecture \ref{characterization} with the theorem of Gao and Habbeger \cite[Theorem 1.3]{Gao:Habegger:RMM}, stated below as Theorem \ref{GH theorem}.  In their setting of abelian varieties, the case of $\dim S = 0$ reduces to the original Manin-Mumford Conjecture proved by Raynaud \cite{Raynaud:2}.  We also introduce a dynamical notion of rank that extends the notion of Betti rank from \cite{ACZ:Betti}.  Conjecture \ref{max rank}, which is equivalent to Conjecture \ref{characterization}, is a characterization of subvarieties of $S\times\bP^N$ with maximal $\Phi$-rank. 

\subsection{$\Phi$-Rank}
Let $\Phi: S\times\bP^N\to S \times\bP^N$ be an algebraic family of morphisms of degree $> 1$, and let $\mathcal{X}\subset S \times \bP^N$ be an irreducible flat subvariety over $S$.  We define the {\bf $\Phi$-rank} of $\mathcal{X}$ to be
$$\rank_{\Phi}(\mathcal{X}):=\max\{r \geq 0~:~\hat{T}_{\Phi}^{\wedge r}\wedge [\mathcal{X}]\not= 0\}.$$
It is clear from the definition that $\rank_{\Phi}(\mathcal{X}) \le \dim\mathcal{X}$.

\begin{lemma} \label{rank equality}
For $\dim S = 0$, so that $\Phi$ is a single endomorphism $f: \bP^N \to \bP^N$ defined over $\C$, and for any  irreducible subvariety $Z$ of $\bP^N$, we have
	$$\rank_f(Z)  = \dim Z.$$
\end{lemma}

\begin{proof}
Let $\ell = \dim Z$. If $\ell=0$, then the statement is clear. So we may assume that $\ell>0$.  Let $\omega$ be the Fubini-Study form on $\bP^N$, representing $c_1(\mathcal{O}(1))$.  If $f$ has degree $d$, then $d^{-n}(f^n)^* \omega$ is cohomologous to $\omega$ for all $n\geq 1$.  In particular, we have 
	$$\int_{\bP^N} \left(d^{-\ell n}(f^n)^* \omega^{\wedge \ell} \right) \wedge [Z]   \; = \; \int_{\bP^N} \omega^{\wedge \ell} \wedge [Z] \; \not= \; 0$$
for all $n\geq 1$.  Because of the local-uniform convergence of the potentials of pullbacks of $\omega$ to that of the invariant current $T_f$, we have 
	$$\left(d^{-\ell n}(f^n)^* \omega^{\wedge \ell} \right) \wedge [Z] \to T_f^{\wedge \ell} \wedge [Z]$$
as $n\to \infty$  \cite[Chapter III Corollary 3.6]{Demailly:CADG}, allowing us to conclude that $T_f^{\wedge \ell} \wedge [Z] \not=0$.  On the other hand, we have $T_f^{\wedge (\ell+1)} \wedge [Z] = 0$ for dimension reasons. 
\end{proof}

We say that $\cX\subset S\times\bP^N$ has {\bf maximal $\Phi$-rank} if $\rank_\Phi(\cX) = r_{\Phi, \cX}$, the relative special dimension $r_{\Phi, \cX}$ of $\cX$, a name justified by the following proposition:

\begin{prop}\label{rank inequality}
If $\mathcal{Y}\subset S \times \bP^N$ is a $\Phi$-special subvariety, then 
$$\hat{T}_{\Phi}^{\wedge (1 + \dim_S \cY) }\wedge [\mathcal{Y}] = 0.$$
In particular, for any irreducible $\mathcal{X}\subset S\times \bP^N$ which is flat over $S$, we have 
 $$\dim_S\cX \leq  \rank_{\Phi}(\mathcal{X})\leq r_{\Phi,\mathcal{X}}.$$
\end{prop}

\begin{proof}
Suppose that $\cY$ is $\Phi$-special, and set $k = \dim_S\cY$.  Let $\Psi$ be a family of polarized endomorphisms over $S$ that commutes with an iterate of $\Phi$ along $\mathcal{Z}$ with $\cY \subset \mathcal{Z} \subset S \times \bP^N$; replacing $\Psi$ with an iterate, we assume that $\cY$ is prefixed for $\Psi$.  Now replace $\cY$ with an iterate under $\Psi$ so that $\Psi(\cY) = \cY$, and pass to a normalization if it is not itself normal.  We have $\hat{T}_\Phi = \hat{T}_\Psi$ on $\cY$.  The current $\hat{T}_\Psi$, for a general family of polarized endomorphisms, is defined similarly to $\hat{T}_\Phi$; see \cite[\S2.3]{Gauthier:Vigny}. 
Since $(\Psi, \cY)$ defines a family of polarized dynamical systems of degree $e>1$, there is a (1,1)-form $\omega$ on $\cY$ representing the first Chern class of the polarizing (relatively ample) line bundle $\cL$ so that $\omega^{k+1} = 0$.  Because of the local-uniform convergence of the potentials of pullbacks of $\omega$ to a potential for $\hat{T}_\Psi$, we know that the wedge power $\hat{T}_\Psi^{\wedge j}$ is the limit of pullbacks $\frac{1}{e^{jn}} (\Psi^n)^* (\omega^{\wedge j})$ on $\cY$, for any $j\geq 1$.  As $\omega^{\wedge(k+1)} = 0$, we have $\hat{T}_{\Phi}^{\wedge (1 + \dim_S \cY) }\wedge [\mathcal{Y}] = 0$.  

The upper bound on $\rank_\Phi(\cX)$ is an immediate consequence, because $\cX$ is contained in a $\Phi$-special subvariety of relative dimension $r_{\Phi, \cX}$.  

For the lower bound on $\rank_\Phi(\cX)$, observe that the slice $(1,1)$-current $T_s$ of $\hat{T}_\Phi$ over $s \in S$ satisfies $T_s^{\wedge \ell} \wedge [X] \not=0$ for every subvariety $X$ of dimension $\ell$ in $\bP^N$ and every $s \in S$, by Lemma \ref{rank equality}.   The nonvanishing of $\hat{T}_\Phi^{\wedge \dim_S\cX} \wedge [\cX]$ follows; see, for example, \cite[Proposition 4.3]{Bassanelli:Berteloot} on slicing.
\end{proof}

Condition (2) of Conjecture \ref{characterization} means that $\rank_\Phi(\cX) \geq r_{\Phi, \cX}$.  By Proposition \ref{rank inequality}, Conjecture \ref{characterization} is thus a characterization of subvarieties $\cX$ of maximal $\Phi$-rank:

\begin{conjecture}\label{max rank}
Let $\Phi: S\times\bP^N\to S \times\bP^N$ be an algebraic family of morphisms of degree $> 1$, and let $\mathcal{X}\subset S\times\bP^N$ be an irreducible subvariety which is flat over $S$. The following are equivalent:
\begin{enumerate}
\item $\mathcal{X}$ contains a Zariski dense set of $\Phi$-preperiodic points. 
\item  
$\rank_{\Phi}(\mathcal{X})=r_{\Phi,\mathcal{X}}$.
\end{enumerate}
\end{conjecture}

\subsection{Families of abelian varieties}
Let us now compare Conjecture \ref{max rank} with the following theorem of Gao and Habegger.

\begin{theorem}\cite[Theorem 1.3]{Gao:Habegger:RMM}  \label{GH theorem}
Let $\mathcal{A}\to S$ be an abelian scheme of relative dimension $g\ge 1$.  Assume that $\mathcal{X}\subset \mathcal{A}$ is a closed irreducible subvariety, flat over $S$, for which the orbit $\Z\cdot \mathcal{X}$ is Zariski dense in $\mathcal{A}$. The following are equivalent:
\begin{enumerate}
\item $\mathcal{X}$ contains a Zariski dense set of torsion points in $\mathcal{A}$. 
\item $\rank_{\mathrm{Betti}}(\mathcal{X})=2g$.
\end{enumerate}
\end{theorem}

\noindent
Here $\rank_{\mathrm{Betti}}(\mathcal{X})$ is the generic Betti rank of $\mathcal{X}$. Its study was initiated by Andr\'e-Corvaja-Zannier \cite{ACZ:Betti} and it has now become ubiquitous in issues related to unlikely intersections.  We briefly recall its definition and then explain why Conjecture \ref{max rank} is a generalization of Theorem \ref{GH theorem}. 

Let $\pi: \mathcal{A}\to S$ denote the projection, and let $\Delta \subset S$ be a simply connected open subset.  Choose holomorphic functions $\omega_i: \Delta\to \bC^g$, $i = 1, \ldots, 2g$, defining a basis of the period lattice of the fibers, so that
	$$\mathcal{A}_s \simeq \bC^g/\omega_1(s)\Z \oplus\cdots\oplus \omega_{2g}(s) \Z$$
for $s \in \Delta$, with the isomorphism denoted by
	$$x \mapsto \sum_{i=1}^{2g}\beta_i(x)\omega_i(s)$$
for $\beta_i(x) \in \R/\Z$.  The {\bf Betti map} $b_{\Delta}:\pi^{-1}(\Delta) \to  \mathbb{R}^{2g}/\mathbb{Z}^{2g}$ is a real-analytic map defined by
$$b_{\Delta}(x)=(\beta_1(x),\ldots, \beta_{2g}(x))$$
for $x\in \mathcal{A}_s$.  The {\bf generic Betti rank} $\rank_{\mathrm{Betti}}(\mathcal{X})$ is  the maximal rank of the differential of the Betti map over all smooth points in $\mathcal{X}\cap \pi^{-1}(\Delta)$, and it is independent of the choice of simply connected $\Delta$.  Associated to the Betti map and a choice of polarization $\cL$ is a $(1,1)$-form $\omega_{\mathcal{A}, \cL}$ on $\mathcal{A}$, called the Betti form and first introduced by Mok \cite{Mok:Kahler}; see \cite[Proposition 2.2]{DGH:uniformity}. It can be defined by  
\begin{equation} \label{Betti form def}
	\omega_{\mathcal{A}, \cL}=\sum_{1\le i<j\le 2g} a_{ij} \, d\beta_i\wedge d\beta_j,
\end{equation}
for constants $a_{ij}\in\bR$; its restriction to each fiber $\mathcal{A}_s$ is the Riemann form associated to the choice of polarization.  See, for example, \cite[\S2.4]{CGHX}.

\begin{prop} \label{GH case}  
Conjecture \ref{max rank} implies Theorem \ref{GH theorem}.
\end{prop}

\begin{proof}
Let us first replace $S$ with a Zariski-open subset so that all fibers of $\cA \to S$ are (smooth) abelian varieties.  Consider the multiplication-by-$M$ morphism
	$$[M]:\mathcal{A}\to\mathcal{A}$$
for any choice of integer $M\geq 2$.  The endomorphism is polarizable with degree $M^2$ on each fiber of the projection $\cA \to S$.  Again replacing $S$ with a Zariski-open subset if necessary, there is an embedding 
	$$\cA \hookrightarrow S \times \bP^N$$
for some $N$, so that, as in \cite[Corollary 2.2]{Fakhruddin:selfmaps}, $[M]$ extends to an algebraic family of morphisms 
	$$\Phi_{M,\cA}:  S \times \bP^N \to S \times\bP^N.$$
The Green current $\hat{T}_{\Phi_{M,\cA}}$ on $S\times\bP^N$ restricts to the Betti form $\omega_{\cA,\cL}$ on $\cA$ associated to the polarization $\cL$; indeed, the current is uniquely determined (in its cohomology class $c_1(\cL)$) by its invariance under pullback for the restriction $\Phi_{M,\cA}|_\cA$, where $[M]^* \omega_{\cA, \cL} = M^2 \omega_{\cA, \cL}$.  

As in the proof of \cite[Proposition 2.2 (iii)]{DGH:uniformity}, we have
\begin{equation}\label{rank equality eq}
	\rank_{\mathrm{Betti}}(\mathcal{X})=2 \rank_{\Phi_{M,\cA}}(\mathcal{X}).
\end{equation}
In detail, let $2k$ be the generic Betti rank of $\cX$; it is always even by, for example, the formulas provided in \cite{ACZ:Betti}.  It is clear that $\rank_{\Phi_{M,\cA}}(\mathcal{X}) \leq k$ from the definition of the Betti form $\omega_{\cA, \cL}$ given in \eqref{Betti form def}.  On the other hand, as pointed out in \cite[\S2.4]{CGHX}, the Betti form acts on pairs of tangent vectors in $\cA$ by the composition of a complex-linear projection to the tangent space to a fiber $\cA_s$ and then applying the K\"ahler form $\omega_s = \omega_{\cA, \cS}|_{\cA_s}$.  In particular, since the image under the projection from $T_x\cX$ to $T_x \cA_s$ at a point $x\in \cX$ is generally a complex subspace of dimension $k$, we deduce that $\omega_{\cA,\cL}^{\wedge k} \not=0$.  In other words, $\rank_{\Phi_{M,\cA}}(\mathcal{X}) \geq k$, demonstrating equality in \eqref{rank equality eq}. 

It remains to relate the hypothesis on $\cX$, that its $\Z$-orbit is Zariski dense in $\cA$, to our notion of relative special dimension. Indeed, the condition on $\cX \subset \cA$ implies that $\cX$ is not contained in any proper subgroup scheme over $S$ nor a torsion-translate of such a subgroup.  In particular, the smallest $\Phi_{M,\cA}$-special subvariety containing $\cX$ in $S\times \bP^N$ is the embedded copy of $\cA$ itself.  That is, $r_{\Phi, \cX} = g$, the relative dimension of $\cA$.
\end{proof}

\begin{remark}
Comparing the statement of Theorem \ref{GH theorem} with Conjectures \ref{characterization} and \ref{max rank}, it is important to note some additional complications that arise in the dynamical setting.  It can happen that the $\Phi$-orbit of a subvariety $\cX \subset S \times\bP^N$ is Zariski dense in some $\Phi$-special subvariety $\mathcal{Y} \subset S\times\mathbb{P}^N$, while $r_{\Phi, \cX} < \dim_S \cY$.  Such examples led to the formulations of the Dynamical Manin-Mumford conjecture in \cite{Ghioca:Tucker:Zhang} and \cite{Ghioca:Tucker:DMM} and the introduction of the auxiliary endomorphism $\Psi$ in the definition of $\Phi$-special.  For example, $\Phi: E \times E \to E\times E$ can act by a product of complex-multiplication endomorphisms on an elliptic curve $E$, chosen so that the diagonal $\Delta_E \subset E\times E$ (which contains a Zariski-dense set of preperiodic points, being the torsion points of $E$) is {\em not} a preperiodic curve for $\Phi$, so its $\Phi$-orbit is Zariski-dense in $E\times E$.  On the other hand, the diagonal $\Delta_E$ invariant under the usual $\Z$-action on $E\times E$, so $r_{\Phi, \Delta_E} = 1$.  See \cite{Ghioca:Tucker:Zhang} and \cite{Pazuki:DMM} for details and explicit constructions.
\end{remark}

\begin{remark}
The statement of Theorem \ref{easy}, that (2) implies (1) in Conjecture \ref{characterization}, is easy to prove in the setting of abelian varieties, due to the smoothness and analytic properties of the Betti forms \cite[Proposition 2.1.1]{ACZ:Betti}.  (The proof in \cite{ACZ:Betti} is written for sections of a family $\mathcal{A} \to S$ of abelian varieties; to treat the case of a subvariety $\cX$ in $\mathcal{A}$, notice that $\mathcal{X}$ is the image of the identity section of the base extension $\mathcal{A}\times _S\mathcal{X}\to \mathcal{X}$.)
\end{remark}

\subsection{Non-degeneracy} \label{non-degeneracy}
We conclude this section by observing that there is an important notion in the literature which is similar to but distinct from condition (2) of Conjecture \ref{characterization}; we include it here for comparison.  By definition, $\cX$ in $S \times\bP^N$ is {\bf non-degenerate for $\Phi$} if $\hat{T}_{\Phi}^{\wedge \dim \cX} \wedge [\cX] \not= 0$.  This concept was introduced for subvarieties $\cX$ in families of powers of elliptic  curves by Habegger \cite{Habegger:special} and played an important role in the proof of the Geometric Bogomolov Conjecture
over function fields in characteristic 0 by Gao--Habegger and Cantat--Gao--Habegger--Xie \cite{Gao:Habegger:Bog, CGHX}, the proof of the uniform Mordell--Lang conjecture by Gao--Habegger, K\"uhne and Gao--Ge--K\"uhne \cite{DGH:pencils, DGH:uniformity, Kuhne:UML, Gao:Ge:Kuhne} and the proof of the relative Manin-Mumford conjecture by Gao--Habegger \cite{Gao:Habegger:RMM}.  In \cite{Yuan:Zhang:quasiprojective}, this notion of non-degeneracy was extended to families of polarized dynamical systems, and it is a key hypothesis for their theorems on arithmetic equidistribution.  In particular, if $\Phi$ and $\cX$ are defined over a number field $K$, then Conjecture \ref{characterization} would characterize the existence of a Zariski-dense set of {\em small} points in $\cX$, while the non-degeneracy condition would imply that the $\Gal(\Kbar/K)$-orbits of these small points are uniformly distributed with respect to the measure $\hat{T}_{\Phi}^{\wedge \dim \cX} \wedge [\cX] \not= 0$ (if the points exist), as proved in \cite[Theorem 6.2.3]{Yuan:Zhang:quasiprojective}  \cite[Theorem 2]{Gauthier:goodheights} \cite[Theorem 6.2]{Gauthier:Taflin:Vigny}.  Note that every projective algebraic subvariety of $\bP^N$ is non-degenerate for a morphism $f: \bP^N\to \bP^N$ when $\dim S = 0$, by Lemma \ref{max rank}.

\bigskip
\section{Some special cases of Conjecture \ref{characterization}}
\label{known}

In this section, we provide examples where Conjecture \ref{characterization} is known, moving away from the setting of abelian varieties.  We begin in the setting of the original Dynamical Manin-Mumford Conjecture, corresponding to the case of $\dim S = 0$, in \S\ref{DMM}.  In \S\ref{N=1}, we observe that Conjecture \ref{characterization} holds in all cases when $N=1$, and we relate it to the theory of $J$-stability for maps on $\bP^1$.  In \S\ref{DAO} we explain that Conjecture \ref{characterization} implies the so-called Dynamical Andr\'e-Oort conjecture (or ``DAO"). Finally, in \S\ref{uniform} we illustrate that uniform versions of the conjecture are consequences of the conjecture itself, with the example of shared preperiodic points for distinct maps on $\bP^1$. 

\subsection{The Dynamical Manin-Mumford Conjecture} \label{DMM}
When $S$ is a point, so that $\dim S = 0$ and $\Phi: \bP^N \to \bP^N$ is an endomorphism over $\C$, Conjectures \ref{characterization} and \ref{dimension} reduce to Zhang's Dynamical Manin-Mumford Conjecture (or DMM), as reformulated by Ghioca and Tucker in \cite{Ghioca:Tucker:DMM}:  conjecturally, a subvariety $X$ of $\bP^N$ contains a Zariski-dense set of preperiodic points if and only if it is a $\Phi$-special subvariety.  The reduction follows from observing that the current $T_\Phi^{\wedge r} \wedge [X]$ will vanish if $r > \dim X$, but is nonzero for $r = \dim X$ by Lemma \ref{rank equality}.  The implication that a $\Phi$-special subvariety of $\bP^N$ always contains a dense set of preperiodic points is well known; see for example \cite{Fakhruddin:selfmaps, Briend:Duval, Dinh:Sibony:allure}.  However, the converse implication has been proved only in a few settings outside of the cases of endomorphisms of abelian varieties, all of which we outline here.

It is worth observing that DMM is obvious for maps on $\bP^1$:  a single point is either preperiodic (which is equivalent to special for points) or it is not.  For $N>1$, the DMM conjecture has been fully resolved for polarized endomorphisms of $(\bP^1)^N$ over $\C$  \cite{Ghioca:Nguyen:Ye:DMM1, Ghioca:Nguyen:Ye:DMM2, Mavraki:Schmidt:Wilms}, for subvarieties $X \subset (\bP^1)^N$ of arbitrary dimension. A proof of DMM for polynomial maps of $\bA^2$ that extend regularly to $\bP^2$, assuming that the complex algebraic curve $X$ satisfies a certain condition on its intersection with the line at infinity, was provided in \cite{Dujardin:Favre:Ruggiero}.  A related problem for polynomial automorphisms of $\mathbb{A}^2$ was treated in \cite{DF}, but Conjecture \ref{characterization} does not cover this setting.

\subsection{Dimension $N=1$} \label{N=1}
Here we show that Conjecture \ref{characterization} in the case of $N=1$ follows from \cite[Theorem 1.1]{D:stableheight}. 

\begin{theorem}  \label{1}
Conjecture \ref{characterization} holds in dimension $N=1$.
\end{theorem}

\begin{proof}
An irreducible flat subvariety $\cX$ of $S \times \bP^1$, if not equal to all of $S \times\bP^1$, is a multi-section over $S$.  That is, after replacing the parameter space $S$ with a branched cover, we may assume that $\mathcal{X}$ is the graph of a marked point $a:  S \to \bP^1$.  If $r_{\Phi, \cX}=0$, then $\cX$ is itself $\Phi$-special, meaning that the marked point $a$ is persistently preperiodic for $\Phi$.  In this case, the preperiodic points are obviously dense in $\cX$ and the current $\hat{T}_\Phi^0 \wedge [\cX] = [\cX]$ is clearly nonzero, so the equivalence of (1) and (2) in Conjecture \ref{characterization} holds.  If $r_{\Phi, \cX} = 1$, then the current of (2) is nonzero if and only if the point is unstable, in the sense of \cite{D:stableheight}.  In other words, the sequence of holomorphic maps $\{s \mapsto \Phi_s^n(a(s))\}$ fails to formal a normal family on the parameter space $S$; see, for example, \cite[Theorem 9.1]{D:lyap} for the equivalence of normality and the vanishing of $\hat{T}_\Phi \wedge [\cX]$.  As proved in \cite[Theorem 1.1]{D:stableheight}, if the point $a$ is not persistently preperiodic, then stability on all of $S$ implies that the family $\Phi$ is isotrivial and the point $a$ will never be preperiodic.  On the other hand, instability implies, by Montel's theory of normal families, that the point $a$ will be preperiodic for a Zariski-dense set of parameters; see, for example, \cite[Proposition 5.1]{D:stableheight}.  So the equivalence in Conjecture \ref{characterization} holds also for $r_{\Phi, \cX} = 1$.
\end{proof}

\begin{remark} \label{equivalent}
Conjecture \ref{characterization} also implies \cite[Theorem 1.1]{D:stableheight} (and therefore also \cite[Theorem 2.5]{Dujardin:Favre:critical} and \cite[Lemma 2.1]{McMullen:families} addressing the case of marked critical points in $\bP^1$), so it is logically equivalent in dimension $N=1$.  Indeed, suppose that $\Phi:S \times\bP^1\to S\times\bP^1$ is an algebraic family of maps of degree $d>1$, and suppose that $\Gamma_a\subset S \times \bP^1$ is the graph of a marked point $a: S \to \bP^1$.  Assume the pair $(\Phi, a)$ is stable, so that $\hat{T}_\Phi\wedge[\Gamma_a]= 0$.   We will deduce from Conjecture \ref{characterization} that $a$ is either persistently preperiodic or the pair $(\Phi,a)$ is isotrivial.  It suffices to assume that $\dim S = 1$.  Passing to a branched cover of $S$ if necessary, we can mark three distinct periodic points for $\Phi$ and, removing a Zariski-closed subset of $S$ where they collide, we can change coordinates on $\bP^1$ so that $\{0,1,\infty\}$ are persistently periodic.  Stability of $(\Phi, a)$ implies one of two things: either (1) the graph $\Gamma_a$ is itself preperiodic (i.e., the special dimension is $r_{\Phi, \Gamma_a} = 0$), or (2) as a consequence of Conjecture \ref{characterization} with $r_{\Phi, \Gamma_a}=1$, the preperiodic points of $\Phi$ cannot be Zariski dense in $\Gamma_a$.  In case (1), we are done.  In case (2), this means that there is a Zariski-open subset $U$ of $S$ over which the point $a$ is never preperiodic.  In particular, the point $a$ and its infinite forward orbit is disjoint from the set $\{0,1,\infty\}$ over the quasiprojective curve $U$.  But, as McMullen observed in his proof of \cite[Lemma 2.1]{McMullen:families}, there are at most finitely many non-constant holomorphic functions $U \to \bP^1\setminus \{0,1,\infty\}$.  This implies that the iterates $\Phi^n(a)$ must be constant (and distinct) over $U$, for all $n$ sufficiently large.  But then, by a simple interpolation argument, we see that $\Phi_s$ is independent of the parameter $s \in U$. In other words, in case (2), the family $\Phi$ is isotrivial.
\end{remark}

\begin{remark} \label{J bif}
Suppose that $\Phi: S \times \bP^1\to S \times \bP^1$ is an algebraic family of endomorphisms on $\bP^1$ and that 
	$$\cX = \Crit(\Phi)$$ 
is the critical locus.  (This $\cX$ is not necessarily irreducible, but we can apply Theorem \ref{1} to each component.)  Then $r_{\Phi, \cX} = 0$ (for all components) if and only if $\Phi$ is a family of postcritically finite maps.  We know from Thurston rigidity that this can hold if and only if $\Phi$ is either an isotrivial family or a family of flexible Latt\`es maps; see \cite[Theorem 1]{Douady:Hubbard:Thurston} or \cite[Theorem 6.2]{McMullen:families}.

For $r_{\Phi, \cX} = 1$, the current of Conjecture \ref{characterization} condition (2),
	$$\hat{T}_\Phi \wedge [\Crit(\Phi)],$$
projects to the parameter space $S$ as the {\bf bifurcation current} $\Tbif$ for the family $\Phi$.  See \cite{D:current} and \cite{Dujardin:Favre:critical} for definitions.  It was proved in \cite{D:current} that the bifurcation current vanishes if and only if the family is $J$-stable in the sense of Ma\~n\'e-Sad-Sullivan \cite{Mane:Sad:Sullivan} and Lyubich \cite{Lyubich:stability}.  So Conjecture \ref{characterization} (i.e., Theorem \ref{1}) implies a (known) characterization of instability in algebraic families $\Phi$ by the existence of many parameters with (non-persistent) preperiodic critical points; compare \cite[Lemma 2.1]{McMullen:families} and \cite[Theorem 2.5]{Dujardin:Favre:critical}.  

We return to the topic of $J$-stability in higher dimensions in Section \ref{stability}.
\end{remark}

\subsection{DAO as a special case} \label{DAO}

Let $f: S\times \mathbb{P}^1\to S\times \mathbb{P}^1$ be an algebraic family of rational maps on $\bP^1$ of degree $d\ge 2$.  We say the family has {\bf dimension $m$ in moduli} if the induced projection from $S$ to $\M_d^1$, the moduli space of all maps of degree $d$ on $\bP^1$ modulo M\"obius conjugacy, has $m$-dimensional image.  Here we observe that Conjecture \ref{characterization} (in fact, in its weaker form of Conjecture \ref{dimension}) implies the following conjecture proposed by Baker-DeMarco \cite{BD:polyPCF} and Ghioca-Hsia-Tucker \cite{Ghioca:Hsia:Tucker}; the statement here is given explicitly in \cite[Conjecture 6.1]{D:stableheight}.

\begin{conjecture} 
\label{DAOonemap}
Assume that $\dim S =m > 0$ and that $f$ has dimension $m$ in moduli.  Let $a_0,\ldots,a_m:  S \to \bP^1$ be $m+1$ marked points, and assume that there is a Zariski-dense set of parameters $s\in S$ such that $a_0(s), \ldots, a_m(s)$ are simultaneously $f_s$-preperiodic.  Then the marked points $a_0, \ldots, a_m$ are dynamically related along $S$. 
\end{conjecture}

By definition, marked points $(a_0, \ldots, a_m)$ are {\bf dynamically related along $S$} if there exists a proper closed subvariety $\cY \subset S \times (\bP^1)^{m+1}$ projecting dominantly to $S$ which is preperiodic for the fiber-product $\Phi = f^{[m+1]}$ defined by 
\begin{equation} \label{Phi product}
\Phi(s,z_0,\ldots,z_m) = (s, f_s(z_0), \ldots, f_s(z_{m}))
\end{equation}
and which contains the graph $\Gamma$ of $(a_0, \ldots, a_m)$ over $S$. 

The converse implication to Conjecture \ref{DAOonemap}, that a dynamical relation implies the density of the simultaneous preperiodic points, was proved in \cite{D:stableheight}.  Special cases of Conjecture \ref{DAOonemap} and closely related results were obtained in \cite{Masser:Zannier, Masser:Zannier:2, BD:preperiodic, BD:polyPCF, DWY:QPer1, DWY:Lattes, Ghioca:Hsia:Nguyen, Ghioca:Hsia:Tucker, GHT:rational, GHT:2parameter, Ghioca:Krieger:Nguyen, GKNY, Ghioca:Ye:cubics, Favre:Gauthier:cubics, Favre:Gauthier:book, Ji:Xie:DAO}.

We wish to emphasize one case of Conjecture \ref{DAOonemap}: if a subvariety $V$ of the moduli space $\M_d^1$ contains a Zariski-dense set of postcritically finite maps, then the subvariety $V$ is conjectured to be ``special", meaning that every $(1 + \dim V)$-tuple of critical points is dynamically related along $V$.  One works with an algebraic family $\Phi$ over a parameter space $S$ that maps dominantly to $V$, with marked critical points $\{c_1, \ldots, c_{2d-2}\}$.  If $m = \dim V$, one considers all subsets of these critical points of size $m+1$.  In this form, the conjecture was dubbed the ``Dynamical Andr\'e-Oort Conjecture" in \cite[Chapter 6]{Silverman:moduli}, or ``DAO" in \cite{Ji:Xie:DAO}, because of parallels between the theory of postcritically finite maps and of elliptic curves with complex multiplication.  This case of the conjecture was resolved for algebraic curves in $\M_d^1$ in \cite{Ji:Xie:DAO} but remains open for higher-dimensional subvarieties of $\M_d^1$.

\begin{theorem}  \label{DAO implication}
Conjecture \ref{dimension} implies Conjecture \ref{DAOonemap}.
\end{theorem}

\begin{proof}
To place Conjecture \ref{DAOonemap} in the setting of Conjectures \ref{characterization} and \ref{dimension}, note that the $\Phi$ of \eqref{Phi product} defines a family of polarized endomorphisms of $(\bP^1)^{m+1}$ and so extends to a family of endomorphisms of projective space $\bP^M$ for some dimension $M\geq m+1$.  In particular, the collection $(a_0, \ldots, a_m)$ is dynamically related if and only if the graph $\Gamma$ of this $(m+1)$-tuple lies in a $\Phi$-special subvariety of relative dimension $\leq m$ over $S$.  Indeed, for powers of nonisotrivial maps on $\bP^1$ such as $\Phi$, all $\Phi$-special subvarieties in $S \times (\bP^1)^{m+1}$ will be preperiodic for $\Phi$; see, for example, the discussion in Section 2 of \cite{Ghioca:Tucker:DMM}, as Latt\`es maps coming from complex multiplication are rigid in $\M_d^1$.

Note also that $\Gamma$ contains a Zariski dense set of $\Phi$-preperiodic points if and only if there is a Zariski dense set of $s\in S$ at which $a_0(s), \ldots, a_N(s)$ are simultaneously $f_s$-preperiodic. 

So assume that Conjecture \ref{dimension} holds for $\cX = \Gamma$ and $\Phi = f^{[m+1]}$, and suppose that the hypotheses of Conjecture \ref{DAOonemap} hold.  Then $\Gamma$ must have codimension $\leq m$ in a $\Phi$-special subvariety.  As $m = \dim S = \dim \Gamma$, we must have $r_{\Phi, \Gamma} \leq m$, so that the graph $\Gamma$ must lie in a $\Phi$-special subvariety of relative dimension $\leq m$.  In other words, the marked points are dynamically related.  
\end{proof}

\subsection{Common preperiodic points in $\bP^1$ and uniform bounds}\label{uniform}
For a map $f: \bP^1\to \bP^1$ let $\Prep(f) \subset \bP^1(\C)$ denote its set of preperiodic points.  It was proved in \cite{BD:preperiodic} and \cite{Yuan:Zhang:II} that any pair of maps $f, g: \bP^1 \to \bP^1$ over $\C$, of degrees $>1$, will satisfy either $\Prep(f) = \Prep(g)$ or $|\Prep(f) \cap \Prep(g)| < \infty$.  This result can be viewed as a special case of Conjecture \ref{characterization}, at least when $\deg f = \deg g$.  Indeed, suppose that $\deg(f) = \deg(g) = d \geq 2$, and consider the action of $\Phi := (f,g)$ on $\bP^1\times\bP^1$.  This $\Phi$ is polarizable and so extends to an endomorphism of some $\bP^N$, restricting to the given map $(f,g)$ on an embedded copy of $\bP^1\times\bP^1$.  The common preperiodic points in $\bP^1$ for $f$ and $g$ correspond to preperiodic points of $\Phi$ in the diagonal $\Delta \subset \bP^1\times\bP^1$.  Conjecture \ref{characterization} reduces to the DMM conjecture in this setting, as discussed in \S\ref{DMM}; explicitly, it implies that $\Delta$ is $\Phi$-special if and only if $|\Prep(f) \cap \Prep(g)| = \infty$.  The equivalence of the equality $\Prep(f) = \Prep(g)$ with $\Delta$ being $\Phi$-special follows from \cite[Theorem 1.4]{Yuan:Zhang:II} combined with \cite[Theorem A]{Levin:Przytycki} and \cite[Theorem 1.10]{Mavraki:Schmidt}.

Conjecture \ref{characterization} (or its weaker form, Conjecture \ref{dimension}) predicts what we view as uniform versions of itself.  For example, in this setting of shared preperiodic points for two maps on $\bP^1$, it implies the existence of the uniform bound that we proved in \cite[Theorem 1.1]{DM:commonprep}:

\begin{prop} \label{uniform shared}  
Conjecture \ref{dimension} implies that for each degree $d\geq 2$, there exists a constant $B_d >0$ and a Zariski-open subset $U_d$ of the space $\Rat_d\times\Rat_d$ of all pairs of maps $f,g: \bP^1 \to \bP^1$ of degree $d$ so that 
$$	|\Prep(f) \cap \Prep(g)| \leq B_d$$
for all $(f,g) \in U_d$.
\end{prop}

\begin{proof}
Fix a degree $d>1$ and let $S_d = \Rat_d \times \Rat_d$ be the space of all pairs of maps on $\bP^1$ of degree $d$, and let  
	$$\Phi :  S_d\times (\bP^1 \times\bP^1) \to S_d \times (\bP^1 \times \bP^1)$$
denote the corresponding algebraic family.  For any integer $k \geq 1$, let $\Phi^{[k]}$ denote the $k$-th fiber power 
	$$\Phi^{[k]} = (f,g, \ldots, f,g)$$ 
on $S_d \times (\bP^1\times\bP^1)^{k}$, so that $\Phi^{[1]} = \Phi$, and set
	$$\cX_{k} = S_d\times\Delta^{k}$$ 
in $S_d \times (\bP^1\times\bP^1)^{k}$ for the diagonal $\Delta \subset \bP^1\times\bP^1$.  Note that the codimension of $\cX_k$ in $S_d \times (\bP^1\times\bP^1)^k$ is equal to $k$, for all $k\geq 1$.

In \cite[Lemma 5.4]{DM:commonprep}, we observed that if the conclusion of the proposition were to fail, then $\cX_{k}$ will contain a Zariski dense set of $\Phi^{[k]}$-preperiodic points for all $k \geq 1$.  So it suffices to show that there exists $k \geq 1$ so that $\cX_{k}$ does not contain a Zariski dense set of $\Phi^{[k]}$-preperiodic points.  

To this end, we will show that the relative special dimension $r_{\Phi^{[k]}, \cX_k}$ is equal to $2k$ for any $k\geq 1$; that is, the subvariety $\cX_k$ of $S_d \times (\bP^1\times\bP^1)^k$ is not contained in any $\Phi^{[k]}$-special subvariety other than $S_d \times (\bP^1\times\bP^1)^{k}$ itself.  Then, taking any $k > \dim S_d$, we obtain our desired conclusion from Conjecture \ref{dimension}.

Let $\mathcal{Z}_k\subset S_d \times (\bP^1\times \bP^1)^k$ be an irreducible ${\Phi}^{[k]}$-special variety that contains $\cX_k$.  It must be that $\mathcal{Z}_k$ is preperiodic for $\Phi^{[k]}$, as the map $\Phi_s$ commutes with nothing but its iterates for a generic choice of $s \in S_d$; see for example \cite[Theorem 1.2]{Ye:symmetries}.  From the structural results on invariant subvarieties for product maps on $(\bP^1)^N$ (for any $N$) over fields of characteristic 0, proved in \cite{Medvedev:Scanlon} (and reproved by different methods in \cite{Ghioca:Nguyen:Ye:DMM1}), it follows that $\mathcal{Z}$ must be all of $S_d \times (\bP^1 \times\bP^1)^{k}$.  That is, $r_{\Phi^{[k]}, \cX_k} = 2k$, and the proof is complete.
\end{proof}

\bigskip
\section{Further consequences of Conjecture \ref{characterization}: sparsity of special subvarieties}
\label{further}

For integers $N\geq 1$ and $d\geq 2$, we let $\M_d^N$ denote the affine complex-algebraic variety which is the space of conjugacy classes of endomorphisms on $\bP^N$ with degree $d$ \cite{Silverman:moduli}.  In this section we show that the following conjecture is a consequence of Conjecture \ref{characterization} (and in fact, of its weaker form, Conjecture \ref{dimension}).  The case where $\cX$ is the critical locus of $\Phi$ is a recent theorem of \cite{Gauthier:Taflin:Vigny}.

\begin{conjecture}\label{sparsity}
Let $N>1$ and $\Phi: S\times \mathbb{P}^N\to S\times \mathbb{P}^N$ be an algebraic family of endomorphisms of degree $d >1$ for which the induced map $S \to \M_d^N$ is dominant. Let $\mathcal{X}\subset S\times \mathbb{P}^N$ be an irreducible flat family of subvarieties over $S$ that is neither all of $S\times\bP^N$ nor has codimension $N$. Then the set of $s\in S(\bC)$ such that $X_s$ is $\Phi_s$-special is not Zariski dense in $S$.
\end{conjecture}

\begin{remark}
Conjecture \ref{sparsity} fails if we do not assume that $S \to \M_d^N$ is dominant.  For example, $\Phi$ could be a family of regular polynomial maps on $\bP^2$ (with coordinates $(x:y:z)$), where the line at infinity $\cX = \{z=0\}$ is invariant for $\Phi_s$, for all $s \in S$.  Even if we further assume that $\mathcal{X}$ is not $\Phi$-special, the conjecture can fail without assuming dominance of $S \to \M_d^N$; for example, the family $\mathcal{X}$ of lines $\{x=s\, y\}$ in $\bP^2$, for $s\in\C$, is not special under $\Phi(s,(x:y:z))=(s, (x^2:y^2:z^2))$ but becomes $\Phi_s$-preperiodic for each root of unity $s$. 
\end{remark}

\begin{theorem} \label{sparsity theorem}  
Conjecture \ref{dimension} implies Conjecture \ref{sparsity}. 
\end{theorem}

\begin{proof}
If for some $s\in S(\bC)$ the subvariety $X_s\subset \mathbb{P}^N$ is $\Phi_s$-special, then $X_s$ must contain a Zariski dense set of preperiodic points (see the discussion in \S\ref{DMM}).  It follows that, for any integer $m\geq 1$, if we consider the product map $\Phi_s^{[m]}= (\Phi_s, \ldots, \Phi_s)$ on $(\bP^N)^m$, then the $\Phi_s^{[m]}$-preperiodic points are also Zariski dense in the product $X_s^m$.  

As in the proof of Proposition \ref{uniform shared}, we consider fiber powers of the family $\Phi$.  Let
	$$\Phi^{[m]}: S\times (\bP^N)^m \to S\times (\bP^N)^m$$
denote the algebraic family formed by the products, for integers $m\geq 1$.  Note that $\Phi^{[m]}$ is a polarizable endomorphism, so it extends as an algebraic family of morphisms
	$$\Phi_{N,m}:  S \times\bP^{D(N,m)} \to S \times\bP^{D(N,m)}$$
for some large $D(N,m)$, restricting to a Zariski-open subset of $S$ if necessary.  The subvariety $S \times (\bP^N)^m$ will sit inside $S \times\bP^{D(N,m)}$ as a $\Phi_{N,m}$-invariant subvariety, so we will restrict our attention to $\Phi^{[m]}$ itself.  

Let $\cX_m$ denote the $m$-th fiber power of $\mathcal{X}$ over $S$, as a subvariety of $S\times (\bP^N)^m$, so each fiber of the projection $\cX_m \to S$ is of the form $X_s^m$.  Let $r_{N,m}$ be the relative special dimension of $\cX_m$, so that 
	$$r_{N,m} := r_{\Phi^{[m]},\cX_m} \leq mN = \dim \, (\bP^N)^m.$$ 
If the set of $s\in S(\bC)$ such that $\mathcal{X}_s$ is $\Phi_s$-special is Zariski dense in $S$, then the $\Phi^{[m]}$-preperiodic points will be Zariski dense in $\cX_m$ for all $m\geq 1$.  Conjecture \ref{dimension} implies that
	$$\dim \cX_m  \geq r_{N,m}$$
for all $m\geq 1$.  We will see, however, that $r_{N,m}=mN$ for all $m$, so that if $m>\dim S$, then 
	$$r_{N,m} > \dim \cX_m=\dim S+ m\dim_S(\mathcal{X}),$$ 
which is a contradiction.  In other words, the conclusion of Conjecture \ref{sparsity} holds.

To see that $r_{N,m} =mN$, we work inductively on $m$ and appeal to a result of Fakhruddin \cite[Theorem 1.2]{Fakhruddin:generic}.  He proved that any irreducible Zariski-closed subvariety $\cY \subset S \times \bP^N$ which projects dominantly to $S$ and is invariant for $\Phi$ must either be all of $S \times \bP^N$ or have codimension $N$.  Moreover, if $\cY$ has codimension $N$, then it must parameterize a finite collection of preperiodic points for $\Phi$. Note that there are no polarizable endomorphisms $\Psi$ commuting with $\Phi$ over all of $S$, except the iterates of $\Phi$ itself, so $\Phi$-special is the same as $\Phi$-preperiodic in this setting.  

For $m=1$, we have $\cX_1:=\cX$ and by  \cite[Theorem 1.2]{Fakhruddin:generic} the only $\Phi$-special subvariety containing $\cX$ is $S\times\bP^N$. So $r_{N,1} =N$.  Now fix $m>1$, and assume that $r_{N,j} = jN$ for all $j<m$.  Suppose that $\mathcal{Z}_m$ is a $\Phi^{[m]}$-special subvariety in $S \times (\bP^N)^m$ containing $\cX_m$.  We aim to show that $\mathcal{Z}_m=S\times(\bP^N)^m$.  

Let 
	$$p: S \times (\bP^N)^m \to S \times (\bP^N)^{m-1}$$ 
denote the projection forgetting the last coordinate $\bP^N$.  By the induction hypothesis, we have $p(\mathcal{Z}_m)= S\times (\bP^N)^{m-1}$, because $p(\mathcal{Z}_m)$ must be $\Phi^{[m-1]}$-special and contain $\cX_{m-1}$.  Let $\mathcal{P}_1 \subset S \times (\bP^N)^{m-1}$ be a subvariety of dimension $= \dim S$ which projects surjectively to $S$ and is pointwise fixed by $\Phi^{[m-1]}$.  (In other words, $\mathcal{P}_1$ is a multisection of the projection from $S \times (\bP^N)^{m-1}$ to $S$, parameterizing a collection of fixed points.)  For each $x \in \mathcal{P}_1$, consider 
	$$\mathcal{Z}^{x}_m := p_m \left(p^{-1}(x) \cap \mathcal{Z}_m\right) \subset \bP^N,$$ 
where $p_m: S \times (\bP^N)^m \to \bP^N$ is the projection to the last coordinate.  Note that 
		$$\Phi^{[m]}(x, y) \; \in \;  p^{-1}(x) \cap \Phi^{[m]}(\mathcal{Z}_m)$$
for every $x \in \mathcal{P}_1$ and $y \in \mathcal{Z}^{x}_m$, because $x$ is a fixed point.  As $\mathcal{Z}_m$ is preperiodic under $\Phi^{[m]}$, and letting $\Phi_{\mathcal{P}_1}$ denote the family $\Phi$ of maps on $\bP^N$ over $\mathcal{P}_1$ via the base change $\mathcal{P}_1 \to S$, we see that $\{\mathcal{Z}^{x}_m: x \in \mathcal{P}_1\}$ defines a family of $\Phi_{\mathcal{P}_1}$-preperiodic subvarieties in $\bP^N$.  Applying Fakhruddin's theorem \cite[Theorem 1.2]{Fakhruddin:generic} to each irreducible component of this family, we see that $\{\mathcal{Z}^{x}_m: x \in \mathcal{P}_1\}$ is either a family of points or all of $\bP^N$ for all $x \in \mathcal{P}_1$.

In the former case, we deduce that $\codim \mathcal{Z}_m = N$ in $S \times (\bP^N)^m$, because the dimension of the fibers of $p|_{\mathcal{Z}_m}$ is upper semi-continuous over $S \times (\bP^N)^{m-1}$.  Again applying \cite[Theorem 1.2]{Fakhruddin:generic}, the irreducibility of $\mathcal{Z}_m$ means that it defines a family of preperiodic points in $\bP^N$ over all of $S \times (\bP^N)^{m-1}$ (viewed as an extended parameter space for the family of maps $\Phi_s$ via the projection $p$).  But the fibers of the projection $p|_{\mathcal{Z}_m}$ over $\cX_{m-1}$ must contain the corresponding fibers of $\cX$, a contradiction. 

We conclude that $\mathcal{Z}^{x}_m = \bP^N$ for all $x \in \mathcal{P}_1$.  Recall that we aim to show that $\mathcal{Z}_m = S \times (\bP^N)^m$.  We now repeat this argument, replacing $\mathcal{P}_1$ with any family $\mathcal{P}$ of marked periodic points for $\Phi^{[m-1]}$ in $S \times (\bP^N)^{m-1}$ and work with an iterate of $\Phi$, and we deduce that the fibers of $\mathcal{Z}_m$ over $\mathcal{P}$ are also all of $\bP^N$. As the periodic points of $\Phi^{[m-1]}$ are Zariski-dense in $S \times (\bP^N)^{m-1}$, and as the fibers of $p|_{\mathcal{Z}_m}$ are equal to $\bP^N$ over this Zariski-dense set, we conclude that every fiber of $p|_{\mathcal{Z}_m}$ is equal to $\bP^N$.  Therefore, $\mathcal{Z}_m = S \times (\bP^N)^m$ and so the relative special dimension of $\cX_m$ is $r_{N,m} = m N$.
\end{proof}

\subsection{PCF density or sparsity in the moduli space}  
Conjecture \ref{sparsity} includes as a special case the sparsity of postcritically finite maps in the moduli space $\M_d^N$ of maps $f: \bP^N\to \bP^N$, for all dimensions $N>1$, as conjectured in \cite{Ingram:Ramadas:Silverman} and proved recently in \cite{Gauthier:Taflin:Vigny}.  That is, it is now known that the set of all (conjugacy classes of) maps for which the critical hypersurface is preperiodic is contained in a proper Zariski-closed subset of the moduli space $\M_d^N$ \cite[Theorem B]{Gauthier:Taflin:Vigny}.  The following statement is an immediate consequence of Theorem \ref{sparsity theorem}, setting $\mathcal{X}=\Crit(\Phi)$ to be the critical locus of the family $\Phi$:

\begin{prop} \label{critical sparsity}
For every degree $d\geq 2$ and dimension $N \geq 2$, Conjecture \ref{dimension} implies that the set of postcritically finite maps $f: \bP^N \to \bP^N$ is contained in a proper Zariski-closed subset of moduli space $\M_d^N$.
\end{prop}

For $N=1$, it is known that the postcritically finite maps form a Zariski-dense subset of the moduli space $\M_d^1$ of maps on $\bP^1$, in any degree $d\geq 2$; see, for example, \cite[Main Theorem]{Buff:Epstein:PCF} or \cite[Theorem A]{D:KAWA}.  This fact can also be seen as a special case of Conjecture \ref{characterization}: 

\begin{prop} \label{critical sparsity}
In dimension $N=1$ and for every degree $d\geq 2$, Conjecture \ref{characterization} implies that the set of postcritically finite maps is Zariski dense in the moduli space $\M_d^1$.
\end{prop}

\begin{proof}
Suppose that $\Phi: S\times\bP^1 \to S \times\bP^1$ is an algebraic family of endomorphisms of $\bP^1$ degree $d >1$ for which the induced map $S \to \M_d^1$ is dominant.  The cardinality of $\Crit(\Phi_s)$ is $\leq 2d-2$ for every $s \in S$.  This implies that 
\begin{equation} \label{inequality in dim 1}
	r_m:= r_{\Phi^{[m]}, \Crit(\Phi)^{[m]}} \leq 2d-2
\end{equation}
for all $m \geq 1$, where $\Crit(\Phi)^{[m]}$ is the $m$-th fiber power of the critical locus, because the elements of the set $\Crit(\Phi_s)^m \subset (\bP^1)^m$ have at most $2d-2$ distinct coordinate entries.  (Note that the diagonal in $\bP^1\times\bP^1$ is invariant for the product map $(\Phi_s, \Phi_s)$.) 

Recall that the bifurcation measure is nonzero on the moduli space $\M_d^1$; indeed, it was first proved in \cite[Proposition 6.3]{Bassanelli:Berteloot} that a rigid Latt\`es map must lie in the support of the measure. It follows that
\begin{equation} \label{relevant current N=1}
	\hat{T}_{\Phi^{[2d-2]}}^{\wedge (2d-2)} \wedge [\Crit(\Phi)^{[2d-2]}] \not=0,
\end{equation}
because this current projects to the bifurcation measure on $\M_d^1$ (via first projecting to the base $S$ and then via the natural map to $\M_d^1$); see \cite[Proposition 1.4]{Gauthier:Taflin:Vigny} for the computation relating the current in \eqref{relevant current N=1} to a bifurcation current in $S$. Let $\cX$ be an irreducible component of $\Crit(\Phi)^{[2d-2]}$ for which $\hat{T}_{\Phi^{[2d-2]}}^{\wedge (2d-2)} \wedge [\cX] \not=0$.  Proposition \ref{rank inequality} then implies that $r_{\Phi^{[2d-2]}, \cX} \geq 2d-2$ for this component; combined with \eqref{inequality in dim 1}, we have equality $r_{2d-2} = 2d-2$.  In particular, the coordinates of points in a fiber of $\cX$ over $S$, in $(\bP^1)^{2d-2}$, are generally distinct.  The implication (2) $\implies$ (1) of Conjecture \ref{characterization} tells us that preperiodic points of $\Phi^{[2d-2]}$ are Zariski-dense in $\cX$.  But, over a Zariski-open and -dense subset $U$ of $S$ where the $2d-2$ critical points are distinct, the existence of a preperiodic point in $\cX$ over $s \in U$ means that each of the $2d-2$ critical points for $\Phi_s$ is preperiodic.  In other words, Conjecture \ref{characterization} implies that the set of postcritically finite maps is Zariski-dense in $\M_d^1$.
\end{proof}

\bigskip
\section{Minimal $\Phi$-rank and stability}
\label{stability}

Let $\Phi: S \times\bP^N \to S \times\bP^N$ be an algebraic family of endomorphisms of degree $>1$.  In this section, we introduce the notion of stability for subvarieties $\cX \subset S\times\bP^N$, following \cite{Gauthier:Vigny}.  We discuss the Geometric Dynamical Bogomolov Conjecture of Gauthier and Vigny \cite[Conjecture 1.9]{Gauthier:Vigny}, which is a conjectural generalization of theorems proved in \cite{Gao:Habegger:Bog, CGHX} for abelian varieties over function fields of characteristic 0.  It is related to our Conjecture \ref{characterization} but is independent except in certain cases.  We follow this discussion with a look at the special case of $J$-stability of $\Phi$ and the bifurcation current introduced in \cite{Bassanelli:Berteloot} and studied further in \cite{Berteloot:Bianchi:Dupont}.

\subsection{Stability} \label{stabilitysub}
Let $\cX\subset S \times \bP^N$ be an irreducible subvariety of codimension $p$ which is flat over $S$. Gauthier and Vigny studied the currents $\hat{T}_{\Phi}^{\wedge k} \wedge [\cX]$, for $k\geq 1$, in \cite{Gauthier:Vigny} and defined an important notion:  the subvariety $\cX$ is said to be {\bf $\Phi$-stable} if 
	$$\hat{T}_{\Phi}^{N+1-p} \wedge [\cX] = 0;$$ 
they prove that this is equivalent to the vanishing of the canonical ${\bf \Phi}$-height of ${\bf X}$ (for the height defined over the function field $\C(S)$, introduced by Gubler \cite{Gubler:local, Gubler:Bogomolov, Gubler:Equidistribution}) \cite[Theorem B]{Gauthier:Vigny}. 

If $\cX$ has codimension 1 in a $\Phi$-special variety $\mathcal{Y}$ but is not $\Phi$-special itself, so that $r_{\Phi, \cX} = \dim_S\cX + 1$, then condition (2) of Conjecture \ref{characterization} is exactly the condition of instability.  The conjecture proposes that instability, in this case, is equivalent to the prevalence of $\Phi$-preperiodic points in $\cX$.  

But for general subvarieties $\cX$, of arbitrary dimension, instability does not guarantee the existence of any preperiodic points, as the following example illustrates.

\begin{example}
Let $S = \C$, and consider the family $\Phi: S \times \bP^2 \to S \times \bP^2$ of degree $d =2$ defined by
	$$\Phi(s, x, y) = (s, x^2 + s, y^2 + s + 10)$$
in affine coordinates $(x,y)$ in $\bA^2 \subset \bP^2$.  Set 
	$$\cX = \{(s, 0, 0):  s \in S\}.$$
Then 
$$\hat{T}_\Phi = dd^c G \quad \mbox{ for } \quad G(s,x,y) = \max\{G_s(x), G_{s+10}(y)\}$$
on $S \times \C^2$, where $G_c(z) := \lim_{n\to\infty} 2^{-n} \log \max\{|f^n_c(z)|, 1\}$ is the escape-rate function for the polynomial $f_c(z) = z^2 + c$.  It follows that 
	$$\hat{T}_\Phi|_{\cX} = dd^c G(s, 0,0).$$
In particular, since $s \mapsto G(s,0,0)$ is subharmonic, nonconstant, and bounded from below, we see that it cannot be harmonic, and therefore $\hat{T}_\Phi \wedge [\cX] \not= 0$.  In other words, the subvariety $\cX$ is unstable over $S$, in the sense of \cite{Gauthier:Vigny}.  On the other hand, we know that $G_c(0) = 0$ if and only if $c$ is in the Mandelbrot set; it follows that $G(s,0,0)>0$ for all $s \in \C$, because there are no parameters $s$ where both $s$ and $s+10$ lie in the Mandelbrot set.  In particular, there are no $\Phi$-preperiodic points in $\cX$.  
Note that $\cX$ is not contained in any nontrivial $\Phi$-special subvarieties of $S \times\bP^2$, so that $r_{\Phi, \cX} = 2$.  We can see immediately that 
	$$\hat{T}_\Phi^{\wedge 2} \wedge [\cX] = 0$$
because $\dim \cX = 1 < 2$, thus supporting the equivalence of Conjecture \ref{characterization}.
\end{example}

\subsection{Minimal $\Phi$-rank} \label{min rank}
From Proposition \ref{rank inequality}, we see that a subvariety $\cX \subset S \times \bP^N$ is $\Phi$-stable if and only if it has {\bf minimal $\Phi$-rank}, meaning that 
	$$\rank_\Phi(\cX) = \dim_S \cX.$$ 
	
\subsection{The Geometric Dynamical Bogomolov conjecture}
Gauthier and Vigny formulated a conjecture that aims to characterize the stable subvarieties $\cX \subset S\times\bP^N$; that is, the subvarieties $\cX$ of minimal $\Phi$-rank.

To formulate their conjecture in our terminology, we let $K = \C(S)$ be the function field of $S$ and introduce a few definitions.  Recall that a family $\Phi: S\times\bP^N\to S\times \bP^N$ is {\bf isotrivial} if, after a base change $S' \to S$, we can change coordinates by a family of automorphisms of $\bP^N$ so that $\Phi_s$ becomes independent of the parameter $s \in S'$.  Similarly, a polarized endomorphism ${\bf \Phi}: {\bf Y} \to {\bf Y}$ defined over $\overline{K}$ is {\bf isotrivial} if there exists a model family over $\C$ with an isotrivial extension to some $\bP^N$.  A subvariety ${\bf W} \subset {\bf Y}$ is $\mathbf{\Phi}$-{\bf isotrivial} if it is itself independent of the parameter in the model over $\C$ after the coordinate change has been made.  An irreducible subvariety $\cX \subset S \times \bP^N$ is said to {\bf come from an isotrivial factor} of $\Phi$ if there exist integers $k_1 \ge 1$ and $k_2\geq 0$, an isotrivial polarized endomorphism ${\bf \Psi}: {\bf Y} \to {\bf Y}$ with $\mathbf{\Psi}$-isotrivial subvariety ${\bf W} \subset {\bf Y}$, a subvariety $\mathbf{Z}\subset \mathbf{P^N}$ defined over $\overline{K}$ with $\mathbf{\Phi}^{k_1}(\mathbf{Z})=\mathbf{Z}$, and a dominant rational map $p: \mathbf{Z}\dashrightarrow \mathbf{Y}$ such that the following diagram commutes
\[ \begin{tikzcd}
\mathbf{Z} \arrow{r}{\mathbf{\Phi}^{k_1}}  \arrow[d,dashed,"p"] & \mathbf{Z}  \arrow[d,dashed,"p"] \\%
\mathbf{Y} \arrow{r}{\mathbf{\Psi}}& \mathbf{Y}
\end{tikzcd}
\]
and so that $\mathbf{\Phi}^{k_2}(\mathbf{X})=p^{-1}(\mathbf{W})$.

\begin{conjecture}  \cite{Gauthier:Vigny} \label{geomBog}
Let $\Phi: S\times\bP^N\to S \times\bP^N$ be an algebraic family of morphisms of degree $> 1$, and let $\mathcal{X}\subset S\times\bP^N$ be an irreducible subvariety which is flat over $S$.  The following are equivalent.
\begin{enumerate}
\item $\widehat{T}^{\wedge \dim_S(\mathcal{X})+1}_{\Phi}\wedge[\mathcal{X}] = 0$
\item Either $\cX$ is $\Phi$-special or it comes from an isotrivial factor. 
\end{enumerate}
\end{conjecture}

As mentioned in \S\ref{stabilitysub}, Gauthier and Vigny proved that condition (1) is equivalent to the vanishing of the canonical height $\hat{h}_{\bf \Phi}({\bf X})$ over the function field $K = \C(S)$ \cite[Theorem B]{Gauthier:Vigny}.  Combining this with Gubler's Inequality \cite[Lemma 4.1, Proposition 4.3]{Gubler:Bogomolov}, we know that, when $\dim S = 1$, condition (1) is also equivalent to the existence of a Zariski-dense set of geometrically small points in ${\bf X}$; that is, the existence of a generic sequence of points ${\bf x}_n \in {\bf X}(\overline{K})$ with $\hat{h}_{\bf \Phi}({\bf x}_n) \to 0$ as $n\to \infty$. 

That (2) implies (1) in Conjecture \ref{geomBog} is straightforward to prove, so the challenge is the conjectural (1) $\implies$ (2).  This implication is known in dimension $N=1$ \cite[Theorem 1.1]{D:stableheight} and for subvarieties $\cX$ with $\dim_S\cX =0$ in arbitrary dimension $N$ \cite[Theorem A]{Gauthier:Vigny}; see also \cite{CH1, CH2}.  Gauthier and Vigny also proved that (1) implies (2) in Conjecture \ref{geomBog} when $\mathbf{\Phi}:\mathbf{P}^2\to\mathbf{P}^2$ is a non-isotrivial polynomial skew-product with an isotrivial first coordinate \cite[Theorem 29]{Gauthier:Vigny:v2}. Mavraki--Schmidt \cite[Theorem 4.1, Theorem 4.3]{Mavraki:Schmidt}, building on their work with Wilms \cite{Mavraki:Schmidt:Wilms}, proved that Conjecture \ref{geomBog} holds when $\dim S = 1$, $\Phi$ is defined over $\Qbar$, and $\mathbf{\Phi}$  preserves a power of the projective line $(\mathbf{P}^1)^{\ell}$ with the additional assumption that if $\ell\ge 3$, $\mathbf{\Phi}$ has no isotrivial factor. In the case of an abelian scheme, that is, where $\Phi$ preserves a family of abelian varieties $\mathcal{A}$ over $S$, inducing a homomorphism, and $\mathcal{X}$ is contained in $\mathcal{A}$, the conjecture was proved in \cite{Gao:Habegger:Bog, CGHX}. 

\begin{remark} \label{minimal or maximal}
Note that Conjecture \ref{geomBog} and Conjecture \ref{characterization} have distinct goals.  Conjecture \ref{geomBog} is a characterization of subvarieties $\cX$ of {\em minimal} $\Phi$-rank, i.e., $\rank_\Phi(\cX) = \dim_S\cX$, while Conjecture \ref{characterization} is a characterization of subvarieties $\cX$ of {\em maximal} $\Phi$-rank, i.e., $\rank_\Phi(\cX) = r_{\Phi, \cX}$.  (Recall that we gave bounds on $\rank_\Phi(\cX)$ in Proposition \ref{rank inequality}.)  A subvariety $\cX$ will have both maximal and minimal $\Phi$-rank if and only if $\dim_S \cX = r_{\Phi, \cX}$, meaning that $\cX$ is itself $\Phi$-special.  And in that case, the subvariety $\cX$ contains a Zariski-dense set of {\em persistently} preperiodic points over $S$, giving rise to a Zariski-dense set of geometric preperiodic points in $\mathbf{X}$.  
\end{remark}

\begin{remark} \label{conjecture overlap}
There is interesting overlap of Conjectures \ref{geomBog} and \ref{characterization}  in the case where $\dim_S\cX = r_{\Phi,\cX} - 1$; that is, where $\cX$ is a family of hypersurfaces within a $\Phi$-special subvariety, but not $\Phi$-special itself.  In this case, Conjecture \ref{geomBog} combined with the Dynamical Manin-Mumford conjecture of \S\ref{DMM} (the $\dim S = 0$ case of Conjecture \ref{characterization}) implies the unknown implication (1) $\implies$ (2) of Conjecture \ref{characterization} for arbitrary $S$; the DMM conjecture is used to eliminate the possibility of an isotrivial factor which is not itself $\Phi$-special but which has a Zariski-dense set of preperiodic points.  But Conjecture \ref{characterization} does not seem to imply the unknown implication (1) $\implies$ (2) of Conjecture \ref{geomBog}, because it does not classify subvarieties where Zariski-density of preperiodic points fails.  
\end{remark}

\subsection{$J$-Stability in any dimension}
Let $\Phi: S\times\bP^N\to S \times\bP^N$ be an algebraic family of morphisms of degree $> 1$ and $\Crit(\Phi)$ the critical locus in $S \times \bP^N$.  Note that the relative dimension of $\Crit(\Phi)$ over $S$ is $N-1$, so for each irreducible component $\cX$ of $\Crit(\Phi)$, the relative special dimension $r_{\Phi, \cX}$ is either $N-1$ or $N$.

Condition (2) of Conjecture \ref{characterization} can be interpreted in terms of $J$-stability of the family, as mentioned in \S\ref{N=1} in the case $N=1$.  Following \cite{Bassanelli:Berteloot} and \cite{Berteloot:Bianchi:Dupont}, we say the family $\Phi$ is {\bf $J$-stable}, if 
	$$\hat{T}_\Phi^{\wedge N} \wedge [\Crit(\Phi)] = 0.$$ 
In dimension $N=1$, one recovers the notion of $J$-stability from \cite{Mane:Sad:Sullivan, Lyubich:stability}; see Remark \ref{J bif}.  If $r_{\Phi,\cX} = N$ for a component $\cX$ of the critical locus, the nonvanishing of the current $\hat{T}_\Phi^{\wedge r_{\Phi, \cX}}\wedge [\cX]$ in condition (2) of Conjecture \ref{characterization} implies instability of the family.  

In \cite{Berteloot:Bianchi:Dupont}, Berteloot, Bianchi, and Dupont proved that preperiodic points of $\Phi$ are dense in the support of $\hat{T}_{\Phi}^{\wedge N}\wedge [\Crit(\Phi)]$ in $S \times \bP^N$ (in the analytic topology), and so Zariski-dense in $\Crit(\Phi)$ itself, thus proving the implication (2) $\implies$ (1) of Conjecture \ref{characterization} in this setting.   In fact, they show that one can find a dense set of points preperiodic to {\em repelling} cycles in the support of $\hat{T}_{\Phi}^{\wedge N}\wedge [\Crit(\Phi)]$; see \cite[Theorem 1.6 and Proposition 3.8]{Berteloot:Bianchi:Dupont}.  The proof of Theorem \ref{easy} that we provide in Section \ref{easy section} follows a similar strategy.  

But the converse implication of Conjecture \ref{characterization}, that (1) $\implies$ (2) for all components $\cX$ of $\Crit(\Phi)$, remains open for every $N>1$.   For a component $\cX$ with $r_{\Phi,\cX} = N-1$, the conclusion is clear from Proposition \ref{rank inequality}.  But for $r_{\Phi, \cX} = N$, we do not know that Zariski-density of preperiodic points in $\cX$ is enough to guarantee instability, except in dimension $N=1$ (see \S\ref{N=1}), because we do not know {\em a priori} that the points are repelling.

Conjecture \ref{characterization} predicts a new characterization of $J$-stability, adding to the known list of characterizations in dimension 1 found in \cite{Berteloot:Bianchi:Dupont} and \cite{Berteloot:Buff}:

\begin{prop} \label{J N}
Conjecture \ref{characterization} implies that an algebraic family $\Phi$ of maps on $\bP^N$ of degree $d>1$ is $J$-stable if and only if each irreducible component $\cX$ of the critical locus $\Crit(\Phi)$ is either $\Phi$-special or the set $\Prep(\Phi) \cap \cX$ is contained in a proper algebraic subvariety of $\cX$.  
\end{prop}

\begin{proof}
Let $\cX$ be an irreducible component of $\Crit(\Phi)$ in $S \times \bP^N$.  As a family of hypersurfaces, its relative special dimension is equal to either $N-1$ or $N$.  If $r_{\Phi, \cX}=N-1$, then it means that $\cX$ is $\Phi$-special.  If $r_{\Phi, \cX} = N$, then the $J$-stability of $\Phi$ implies that $\hat{T}_\Phi^{\wedge N} \wedge [\cX] = 0$.  The implication (2) $\implies$ (1) of Conjecture \ref{characterization} implies that preperiodic points of $\Phi$ cannot be Zariski-dense in $\cX$.  
\end{proof}

Finally, we remark that for $N>1$, we do not have a classification of families $\Phi$ for which the latter condition of Proposition \ref{J N} happens, i.e., when $\Prep(\Phi) \cap \cX$ is not Zariski-dense in a component $\cX$ of $\Crit(\Phi)$.  The Geometric Bogomolov Conjecture \ref{geomBog} above aims to rectify this, proposing that such an $\cX$ must come from an isotrivial factor.

\bigskip
\section{Forced intersections: the proof of Theorem \ref{easy}}
\label{easy section}

Throughout this section, we will assume that condition (2) of Conjecture \ref{characterization} holds.  That is, we let $\Phi: S\times\bP^N\to S \times\bP^N$ be an algebraic family of morphisms of degree $> 1$ and $\mathcal{X}\subset S\times\bP^N$ an irreducible subvariety which is flat over $S$. Let $\cY$ be a $\Phi$-special subvariety for $\Phi$ of relative dimension $r_{\Phi, \cX}$ that contains $\cX$, which is preperiodic for an endomorphism $\Psi$ that commutes with $\Phi$ on $\cY$.  Recall that the preperiodic points of $\Psi$ coincide with those of $\Phi$ in $\cY$.  We assume that 
	$$T := \hat{T}_{\Phi}^{\wedge r_{\Phi,\mathcal{X}}}\wedge [\mathcal{X}]\not= 0$$
as a current on $S \times \bP^N$.  

Recall that $r_{\Phi, \cX} \geq \dim_S \cX$ from Proposition \ref{rank inequality}, and let us first assume that $r_{\Phi, \cX} = \dim_S \cX$.  This means that $\cX$ is itself $\Phi$-special, so it is preperiodic for $\Psi$.  If $\dim_S \cX = 0$, then every point in $\cX$ is preperioidc for $\Phi$.  If $\dim_S \cX >0$, then, passing to a forward iterate $\cX'$ that is periodic for $\Psi$, the periodic points will be Zariski dense in every fiber $(\cX')_s$ over $s \in S$, as proved in \cite{Fakhruddin:selfmaps, Briend:Duval, Dinh:Sibony:allure}.  Therefore the preperiodic points of $\Phi$ will be Zariski dense in $\cX$, and so condition (1) of Conjecture \ref{characterization} holds in this case.

Now assume that $r_{\Phi, \cX} > \dim_S \cX$, so that $\cX$ is a proper subvariety of the $\Phi$-special $\cY$.  The slices of $\hat{T}_\Phi^{\wedge r_{\Phi, \cX}}$ over $S$ are measures of maximal entropy for the restriction of $\Psi$ to (an iterate of) $\cY$.  Working with an iterate of $\Psi$ instead of $\Phi$, and with forward iterates of $\cX$ and $\cY$, we will assume that $\cY$ is fixed by $\Phi$ itself, so that each vertical slice of $\hat{T}_\Phi^{\wedge r_{\Phi, \cX}}$ is the measure of maximal for $\Phi$ in $\cY_s$, for $s \in S$.  

To prove that (2) $\implies$ (1) in Conjecture \ref{characterization}, we follow the proof strategy of \cite[Theorem 0.1]{Dujardin:higherbif}, \cite[Theorem 1.6 and Proposition 3.8]{Berteloot:Bianchi:Dupont}, and \cite[Theorem 2.2]{Gauthier:abscont}.

Consider the nonzero current $T_S = \pi_*T$ on $S$, where $\pi: S \times\bP^N \to S$ is the projection.  Fix $\lambda_0$ in the support of $T_S$, and choose a hyperbolic repelling set $K_0$ for $\Phi_{\lambda_0}$ in $\cY_{\lambda_0}$; it moves holomorphically over a neighborhood $U$ of $\lambda_0$ in $S$ \cite[Theorem C]{Jonsson:motion}.  We may select $K_0$ so that it supports a probability measure $\nu$ with maximal entropy for the restriction of $\Phi_{\lambda_0}$ to $K_0$ and which is a measure of type PLB; see, for example, \cite[Theorems 3.2.1 and 3.9.5]{Dinh:Sibony:allure}, \cite[Theorems 3.9 and 3.11]{Berteloot:Bianchi:Dupont}, and \cite[Proposition 1.5]{Gauthier:abscont}.  In particular, $\nu$ puts no mass on proper analytic subvarieties of $\cY_{\lambda_0}$, and so its support is Zariski dense in $\cY_{\lambda_0}$.  Let 
	$$r := r_{\Phi, \cX} = \dim_S \cY$$ 
denote the relative dimension of $\cY$. For each $z \in K_0$, let $\Gamma_z$ denote the graph of the motion of $z$ over the open set $U$ in $S$, and define 
	$$\hat{\nu} = \int_{K_0} [\Gamma_z] \, d\nu.$$
Then $\hat{\nu}$ is a positive and uniformly laminar $(r,r)$-current on 
 	$$\cY_{U} := \pi^{-1}(U) \cap \cY.$$

Observe that 
	$$\frac{1}{d^{n r}} (\Phi|_{\cY}^n)^* \hat{\nu} \longrightarrow \hat{T}_\Phi^{\wedge r}$$
by the characterization of $\hat{T}_\Phi$ and the continuity of its potentials.  On the other hand, we know that $\hat{\nu} = dd^c V$ for a locally bounded current of type $(r-1, r-1)$ since $\nu$ is a PLB measure, as explained in the proof of \cite[Theorem 2.6]{Gauthier:abscont}.  It follows that 
	$$\frac{1}{d^{n r}} (\Phi^n)^* \hat{\nu} \wedge [\cX] \longrightarrow \hat{T}_\Phi^{\wedge r} \wedge [\cX] = T > 0.$$

Now consider the iterated images of $\cX$ over $U$; let $\cX_n = \Phi^n(\cX)$.  Since $T\not= 0$, we know that $\dim \cX \geq r$; let $\chi_U$ be a smooth function with compact support in  $U$ which is $\equiv 1$ in a neighborhood of $\lambda_0$.  Set 
	$$\beta = (\chi_U\circ \pi) \, \hat{\omega}^{\dim \cX - r}$$
on $\cY_U$, where $\hat\omega$ is a smooth $(1,1)$-form on $S\times\bP^N$ restricting to Fubini-Study on each fiber $\bP^N$.   Then 
\begin{eqnarray*}
\int_{\cY_U} (\Phi^n)^* \hat{\nu} \wedge [\cX]\wedge \beta 
	&=& \int_{\cY_U} \hat{\nu} \wedge (\Phi^n)_* [\cX] \wedge (\Phi^n)_* \beta \\
	&=& c_n\, \int_{\cY_U} \hat{\nu} \wedge [\cX_n] \wedge (\Phi^n)_* \beta 
\end{eqnarray*}
for some $c_n > 0$.  In particular, the current $\hat{\nu} \wedge [\cX_n]$ will be nonzero in $\cY_U$ for all sufficiently large $n$.  But by \cite[Theorem 3.1]{Dujardin:higherbif}, we know that
	$$\hat{\nu} \wedge [\cX_n] = \int_{K_0} [\Gamma_z] \wedge [\cX_n] \, d\nu,$$
and the intersections of $\cX_n$ with a positive $\nu$-measure set of graphs $\Gamma_z$ will be transverse, over a small neighborhood of $\lambda_0$.  Since the repelling periodic points are dense in $K_0$, it follows that $\cX_n$ must intersect the graphs of repelling points, in a set which is dense in a set of positive $\nu$-measure.  Moreover, these intersections are Zariski dense in $\cX_n$.  We conclude that the preperiodic points are Zariski dense in $\cX$. 

This completes the proof of Theorem \ref{easy}. \qed

\bigskip \bigskip

\def\cprime{$'$}

\bigskip\bigskip

\end{document}